\documentclass{amsproc}
\usepackage{amsfonts}
\usepackage[utf8]{inputenc}
\usepackage[T1]{fontenc}

\theoremstyle{plain}
\newtheorem{acknowledgement}{Acknowledgement}

\newtheorem{corollary}{Corollary}

\newtheorem{definition}{Definition}
\newtheorem{example}{Example}

\newtheorem{lemma}{Lemma}

\newtheorem{proposition}{Proposition}
\newtheorem{remark}{Remark}

\newtheorem{theorem}{Theorem}
\input{tcilatex}

\begin{document}
\title[The logarithmic Cauchy quotient means]{The logarithmic Cauchy
quotient mean}
\author{Martin Himmel}
\address{Technical University Mountain Academy Freiberg, Faculty of
Mathematics and Computer Science, Intitute of Applied Analysis, Nonnengasse
22, 09596 Freiberg, Germany}
\email{M.Himmel@wmie.uz.zgora.pl}
\author{Janusz Matkowski}
\curraddr{Institute of Mathematics,
University of Zielona G\'{o}ra, Szafrana 4A, PL 65-516 Zielona G\'{o}ra,
Poland}
\email{J.Matkowski@wmie.uz.zgora.pl}

\begin{abstract}
Motivated by recent results on beta-type functions, a new family of means,
which are of logarithmic Cauchy quotient type, are determined and characterized.
\end{abstract}

\maketitle

\section{Introduction}

\footnotetext{\textit{2010 Mathematics Subject Classification. }Primary:
26E60, 39B12.
\par
\textit{Keywords and phrases:} mean, premean, logarithmic Cauchy quotient
mean, functional equation.
\par
{}}

The relationship between the Euler Gamma function and the Beta function
inspired to introduce the beta-type function \cite{MatHim4}. Here we propose
the $k$-variable logarithmic Cauchy quotients, the logarithmic counterpart
of beta-type functions, as follows. Given a positive integer $k\geq 2$, and
a function $f:I\rightarrow \left( 0,+\infty \right) $ (or $f:I\rightarrow
\left( -\infty ,0\right) $) where $I\subset \left( 0,+\infty \right) $ is an
interval that is closed under multiplication, we define the $k$-variable
logarithmic Cauchy quotient $L_{f,k}:I^{k}\rightarrow
\left( 0,+\infty \right) $ by%
\begin{equation}
L_{f,k}\left( x_{1},...,x_{k}\right) =\frac{f\left( x_{1}\right)
+\cdots+f\left( x_{k}\right) }{f\left( x_{1} \cdots  x_{k}\right) },
\label{eq:Lfk}
\end{equation}%
and we refer to $f$ as its generator (Section 2, Definition 1). Similarly to
the case of beta-type functions (\cite{MatHim4}), we give conditions under
which $L_{f,k}$ is a premean or a mean (see Lemma 2, Theorem 2, Theorem 3 and Theorem
4).

In Section 3, assuming that $1\in I$, we prove that two $k$-variable
logarithmic Cauchy quotients coincide if and only if their generators are
proportional (Theorem 1).

In Section 4, applying the theory of iterative functional equations \cite%
{Kuczma}, we determine the general solution of the functional equation%
\begin{equation}
f\left( x\right) =\frac{x}{k}f\left( x^{k}\right) \text{,}
\label{eq:Lfk_Reflexivity}
\end{equation}%
that is the reflexivity condition of $L_{f,k}$ (Lemma 2). Based on this
lemma, in Section 5, we prove Theorem 2, our main result, which says that $%
L_{f,k}$ is a $k$-variable mean in $\left( 1,+\infty \right) $\ iff 
there is $c\neq 0$ such that 
\begin{equation}
f\left( x\right) =c\frac{\log x}{\sqrt[k-1]{x}}, \label{eq:meanGenLfk}
\end{equation}%
for all $x \in (1, +\infty)$, or, equivalently, that $L_{f,k}=\mathcal{L}_{k},$ where $%
\mathcal{L}_{k}$ is a new $k$-variable mean, called the $k$-variable \textit{%
logarithmic} \textit{Cauchy quotient mean }(Definition 2), and defined by 
\begin{equation}
\mathcal{L}_{k}\left( x_{1},\ldots,x_{k}\right) :=\sum\limits_{i=1}^{k}\frac{%
\log x_{i}}{\sum\limits_{l=1}^{k}\log x_{l}}\mathcal{G}_{k-1}\left(
x_{1},\ldots,x_{i-1},x_{i+1},\ldots,x_{k}\right),
\label{eq:Lfk_mean}
\end{equation}%
for all $x_{1},\ldots,x_{k} \in (1, + \infty)$,
where $\mathcal{G}_{k-1}:\left( 0,+\infty \right) ^{k-1}\rightarrow \left(
0,+\infty \right) $ is the $\left( k-1\right) $-variable geometric mean,%
\begin{equation*}
\mathcal{G}_{k-1}\left( x_{1},\ldots,x_{k-1}\right) =\sqrt[k-1]{x_{1}\cdots x_{k-1}}\text{.}
\end{equation*}%
Moreover, some properties of $\mathcal{L}_{k}$ are discussed, the results
for the interval $\left( 0,1\right) $ is formulated, as well as the
corresponding extension of the logarithmic Cauchy quotient mean on the interval $%
\left( 0,+\infty \right) $ (denoted by $\mathfrak{L}_{k}$) is proposed.

We end our paper with two characterizations of the mean $\mathcal{L}_{k}$.
In section 6, applying a variant of the Krull theorem on difference equations (\cite{Krull}) given in Kuczma \cite{{Kuczma}},
we show that  $L_{f,k}=%
\mathcal{L}_{k}$ iff $L_{f,k}$ is reflexive in $\left( 1,+\infty \right) $ and the
function $\log \circ f\circ \exp \circ \exp $ is convex. 
\ In Section 7, assuming that $f:\left( 1,+\infty \right)
\rightarrow \left( 0,+\infty \right) $ is extendable to a function of the
class $C^{2}$ in $\left[ 1,+\infty \right) $, we prove that $L_{f,k}$ is a
premean in $\left( 1,+\infty \right) $ iff it coincides with the mean $%
\mathcal{L}_{k}$.

\section{Some basic notions}

Throughout this paper $I\subset \mathbb{R}$ stands for an interval.

Let $k\in \mathbb{N}$, $k\geq 2.$ A function $M:I^{k}\rightarrow \mathbb{R}$
is called a $k$-\textit{variable mean in} $I$, if%
\begin{equation*}
\min \left( x_{1},\ldots ,x_{k}\right) \leq M\left( x_{1},\ldots
,x_{k}\right) \leq \max \left( x_{1},\ldots ,x_{k}\right) ,\text{ \ \ \ \ }%
x_{1},\ldots ,x_{k}\in I;
\end{equation*}%
and it is called \textit{strict}, if these inequalities are strict for all
nonconstant $k$-tuples $\left( x_{1},\ldots ,x_{k}\right) \in I^{k}.$

Let us note the following easy to verify properties of means.

\begin{remark}
\label{PropertiesOfAMean}
If $M$ is a $k$-variable mean in an interval $I,$ then
\begin{enumerate}
	\item[(i)] \ \ for every subinterval $J\subset I$, $M$ restricted to $J^{k}$ is a
mean in $J$, and $M\left( J^{k}\right) =J,$ in particular, $%
M:I^{k}\rightarrow I;$

\item[(ii)] $\ M$ is reflexive, i.e.%
\begin{equation*}
M\left( x,...,x\right) =x,\ \ \ \ \ x\in I.
\end{equation*}
\end{enumerate}

\end{remark}

A function $M:I^{k}\rightarrow \mathbb{R}$ is called $k$-variable \textit{%
premean} in $I,$ if it reflexive and $M\left( I^{k}\right) =I$ (see \cite%
{JM2006}, also \cite{Toader}, p. 29).

\begin{remark}
If a reflexive function $M:I^{k}\rightarrow \mathbb{R}$ is (strictly) increasing in
each variable, then it is a (strict) $k$-variable mean in $I$.
\end{remark}

Let us introduce some notion playing here a significant role.

\begin{definition}\label{def:Lfk}
Let $k\in \mathbb{N}$, $k\geq 2,$ be fixed, and let $I\subset \left(
0,+\infty \right) $ be an interval that is closed under multiplication. For a
 function $f:I\rightarrow \left( 0,+\infty \right) $ (or $%
f:I\rightarrow \left( -\infty ,0\right) $), the function $%
L_{f,k}:I^{k}\rightarrow \left( 0,+\infty \right) $ defined by \eqref{eq:Lfk}
is called \textit{\ }$k$\textit{-variable logarithmic Cauchy quotient},%
\textit{\ and }$f$ is called a\textit{\ generator of }$L_{f,k}$.
\end{definition}

\begin{remark}
An open interval $I \subset \mathbb{R}$ is closed under multiplication iff
$I=(p, +\infty)$ for some $p \in [1,+\infty)$; or $I=(0, p)$ for some $p \in (0,1]$, or $I=\mathbb{R}$.
\end{remark}

From the definitions of the logarithmic Cauchy quotient $L_{f,k}$ and the
reflexivity we obtain

\begin{remark}
\label{LfkReflexivity}
Under the assumptions of this definition, the logarithmic Cauchy quotient $%
L_{f,k}:I^{k}\rightarrow \left( 0,+\infty \right) $ of a generator $%
f:I\rightarrow \left( 0,+\infty \right) $ is reflexive (or it is a mean or a premean) if its generator $f$ satisfies the iterative
functional equation \eqref{eq:Lfk_Reflexivity}.
\end{remark}

\section{Equality of two logarithmic Cauchy quotients and a functional
equation}

\begin{remark}
Let $k\in \mathbb{N}$, $k\geq 2$,  and interval $I\subset \left( 0,+\infty
\right) $ satisfy conditions of Definition 1 and let $f,g:I\rightarrow
\left( 0,+\infty \right) $. Then $L_{g,k}=L_{f,k}$ iff the functions\ $f$ and 
$g$ satisfy the functional equation%
\begin{equation}
\label{eq:Lfk=Lgk}
\frac{g\left( x_{1}\cdots  x_{k}\right) }{f\left( x_{1}\cdots 
x_{k}\right) }=\frac{g\left( x_{1}\right) +\cdots +g\left( x_{k}\right) }{%
f\left( x_{1}\right) +\cdots+f\left( x_{k}\right) }\text{, \ \ \ \ \ \ }%
x_{1},\ldots,x_{k}\in I.  
\end{equation}%
Moreover, if \thinspace $1\in I$, then $f$ and $g$ satisfy this equation if,
and only if, $g=cf$ for some $c>0$.
\end{remark}

\begin{proof}
The first fact is an immediate consequence of Definition 1. To show the
remaining one, assume that $f$ and $g$ satisfy this equation. Putting $%
x_{1}=x$ and $x_{2}=x_{3}=\ldots=x_{k}=1$ gives%
\begin{equation*}
\frac{g\left( x\right) }{f\left( x\right) }=\frac{g\left( x\right) +\left(
k-1\right) g\left( 1\right) }{f\left( x\right) +\left( k-1\right) f\left(
1\right) },\ \ \ \ \ \ x\in I,
\end{equation*}%
whence 
\begin{equation*}
\left( k-1\right) f\left( 1\right) g\left( x\right) =\left( k-1\right)
g\left( 1\right) f\left( x\right) \text{, \ \ \ \ \ }x\in I.
\end{equation*}%
Since $f$ and $g$ are positive functions,
it follows that $f\left( 1\right) \neq 0$ and $g\left( 1\right) \neq 0$.
Setting $c:=\frac{g\left( 1\right) }{f\left( 1\right) }$ we hence get $g=cf$.
The converse implication is obvious.
\end{proof}

In the sequel we have to exclude $1$ from the interval $I$, as we are mainly
interested in the case when $f\left( 1\right) =0=g\left( 1\right) $. It
turns out that in this case the above functional equation is not trivial. We
prove

\begin{lemma}
Let\ $k\in \mathbb{N}$, $k\geq 2$, be fixed. If the functions $f,g:\left(
1,+\infty \right) \rightarrow \left( 0,+\infty \right) $ \verb|[|or $f,g:\left(
0,1\right) \rightarrow \left( 0,+\infty \right) $\verb|]| satisfy equation \eqref{eq:Lfk=Lgk} with $%
I=\left( 1,+\infty \right) $ \ \verb|[|or $I=\left( 0,1\right) $\verb|]|, and 
\begin{equation*}
c:=\lim_{x\rightarrow 1}\frac{g\left( x\right) }{f\left( x\right) }\text{ }
\end{equation*}%
exists, then $g=cf$.
\end{lemma}

\begin{proof}
Put $h:=\frac{g}{f}$. Setting $x_{1}=x_{2}=...=x_{k}=x$ in \eqref{eq:Lfk=Lgk}, we get $%
h\left( x^{k}\right) =h\left( x\right) $ for all$\ x\in(1,+\infty)$, or equivalently,%
\begin{equation*}
h\left( x\right) =h\left( x^{\frac{1}{k}}\right) ,\ \ \ \ \ \ x\in \left(
1,+\infty \right) ,
\end{equation*}%
whence, by induction,%
\begin{equation*}
h\left( x\right) =h\left( x^{\frac{1}{k^{n}}}\right) ,\ \ \ \ \ \ n\in 
\mathbb{N}\text{, \ \ }x\in \left( 1,+\infty \right) .
\end{equation*}%
Letting $n\rightarrow +\infty $ we hence get $h\left( x\right) =c$ for all $%
x\in \left( 1,+\infty \right) .$
\end{proof}

From this lemma we obtain

\begin{theorem}
Let\ $k\in \mathbb{N}$, $k\geq 2$, be fixed. Assume that $f,g:\left(
1,+\infty \right) \rightarrow \left( 0,+\infty \right) $ (or $f,g:\left(
0,1\right) \rightarrow \left( 0,+\infty \right) $ are such that the limit $%
\lim_{x\rightarrow 1}\frac{g\left( x\right) }{f\left( x\right) }$ exists.
Then $L_{f,k}=L_{g,k}$ if and only if $g=cf$ for some $c>0$.
\end{theorem}

\section{Reflexivity of the logarithmic Cauchy quotient}

Applying the theory of the iterative functional equations (see\ \cite{Kuczma}%
, p. 46, Theorem 2.1) one gets

\begin{lemma}
Fix an integer $k \geq 2$ and $p \in (1, + \infty)$. Then
\begin{enumerate}
	\item[(i)] 
a function $f:\left[ p,+\infty \right) \rightarrow \left( 0,+\infty
\right) $ satisfies equation \eqref{eq:Lfk_Reflexivity} for all $x \in \left[ p,+\infty \right)$
if and only if 
\begin{equation}
f\left( x\right) =k^{n}x^{^{\frac{k^{-n}-1}{k-1}}}f_{0}\left(
x^{k^{-n}}\right)
\label{eq:Lfk_Reflexivity_Solution}
\end{equation}%
for all $x\in \left[ p^{k^{n}},p^{k^{n+1}}\right)$ and $n\in \mathbb{N}_{0}$,
where $f_{0}:=f\mid _{_{\left[ p,p^{k}\right) }};$ moreover, $f$ is
continuous if and only if so is $f_{0}$ and 
\begin{equation}
\lim_{x\rightarrow p^{k}-}f_{0}\left( x\right) =\frac{k}{p}f_{0}\left(
p\right) 
\label{eq:Lfk_Reflexivity_Solution_Continuity}.
\end{equation}

\item[(ii)] a function $f:\left( p,+\infty \right) \rightarrow \left( 0,+\infty
\right) $ satisfies equation \eqref{eq:Lfk_Reflexivity} for all $x \in (p, + \infty)$
if and only if condition \eqref{eq:Lfk_Reflexivity_Solution}
holds for all $x\in \left[ p^{k^{n}},p^{k^{n+1}}\right)$ and $n\in \mathbb{N}_{0}$,
where $f_{0}:=f\mid _{_{\left[ p,p^{k}\right) }};$ moreover, $f$ is
continuous if and only if so is $f_{0}$ and \eqref{eq:Lfk_Reflexivity_Solution_Continuity} holds true.

\item[(iii)] a function $f:\left( 1,+\infty \right) \rightarrow \left( 0,+\infty
\right) $ satisfies equation \eqref{eq:Lfk_Reflexivity} for all $x \in (1, + \infty)$
if and only if condition \eqref{eq:Lfk_Reflexivity_Solution} holds
for all $x\in \left[ p^{k^{n}},p^{k^{n+1}}\right)$ and $n\in \mathbb{Z}$,
%
where $f_{0}:=f\mid _{_{\left[ p,p^{k}\right) }};$ moreover, $f$ is
continuous if and only if so is $f_{0}$ and \eqref{eq:Lfk_Reflexivity_Solution_Continuity} holds true.
\end{enumerate}

\end{lemma}

\section{Means of the logarithmic Cauchy quotient type}

\begin{definition}
The function $\mathcal{L}_{k}:\left( 1,+\infty \right) ^{k}\rightarrow \left(
1,+\infty \right) $, given by%
\begin{equation}
\mathcal{L}_{k}\left( x_{1},\ldots ,x_{k}\right) :=\sum\limits_{i=1}^{k}%
\frac{\log x_{i}}{\sum\limits_{l=1}^{k}\log x_{l}}\left(
\prod\limits_{j=1,j\neq i}^{k}x_{j}\right) ^{\frac{1}{k-1}},
\label{eq:Lfkmean}
\end{equation}%
 that is, 
\begin{equation*}
\mathcal{L}_{k}\left( x_{1},\ldots ,x_{k}\right) =\sum\limits_{i=1}^{k}\frac{%
\log x_{i}}{\sum\limits_{l=1}^{k}\log x_{l}}\mathcal{G}_{k-1}\left(
x_{1},\ldots,x_{i-1},x_{i+1},\ldots,x_{k}\right) 
\end{equation*}%
for all $x_{1},\ldots ,x_{k} \in (1,+\infty)$, where $\mathcal{G}_{k-1}$ is the $\left( k-1\right) $-variable symmetric
geometric mean in $\left( 1,+\infty \right) $,
is called $k$-variable logarithmic Cauchy quotient mean in $\left( 1,+\infty \right) $.

\end{definition}

We also use some elementary fact on the Jensen equation of two or more variables.

\begin{lemma}
\label{kJensen}
\textit{Let }$C$\textit{\ be a convex set of a linear space. A function }$%
f:C\rightarrow \mathbb{R}$\textit{\ is a Jensen function of }$k$\textit{\ variables
for some }$k\in \mathbb{N},$\textit{\ }$k\geq 2$\textit{, i.e., it satisfies the
equality}%
\begin{equation}
f\left( \frac{x_{1}+\cdots+x_{k}}{k}\right) =\frac{f\left( x_{1}\right)
+\cdots+f\left( x_{k}\right) }{k}\text{, \ \ \ }x_1,\ldots,x_k\in C,
\label{thm:kJensen}
\end{equation}%
\textit{if and only if it is a Jensen function of two variables, i.e.,} 
\begin{equation*}
f\left( \frac{x+y}{2}\right) =\frac{f\left( x\right) +f\left( y\right) }{2}%
\text{, \ \ \ }x,y\in C.
\end{equation*}
\end{lemma}

\begin{proof}
\bigskip Indeed, for arbitrary $x,y\in C,$ using \eqref{thm:kJensen}, we have%
\begin{equation*}
f\left( \frac{x+y}{2}\right) =f\left( \frac{x+y+\sum\limits_{i=1}^{k-2}\frac{x+y}{2}%
}{k}\right) =\frac{f\left( x\right) +f\left( y\right)
+\sum\limits_{i=1}^{k-2}f\left( \frac{x+y}{2}\right) }{k},
\end{equation*}%
whence $f\left( \frac{x+y}{2}\right) =\frac{f\left( x\right) +f\left(
y\right) }{2}$, so $f$ is a Jensen function of two variables.

If $f$ is a Jensen function of two variables, then (see Kuczma \cite{Kuczma2}, p. 126,
Lemma 1, where the Jensen convexity is considered), by induction, for every $%
n\in \mathbb{N}$, we get%
\begin{equation*}
f\left( \frac{x_{1}+\cdots+x_{2^{n}}}{2^{n}}\right) =\frac{f\left( x_{1}\right)
+\cdots+f\left( x_{2^{n}}\right) }{2^{n}}\text{, \ \ \ \ \ \ }%
x_{1},\ldots,x_{2^{n}}\in C\text{. }
\end{equation*}%
Let \ $x_{1},\ldots,x_{k}\in C$ be arbitrarily fixed$.$ Choosing $n$ such that $%
k\leq 2^{n}$ and setting here 
\begin{equation*}
x_{k+1}=x_{k+2}=\cdots=x_{2^{n}}:=\frac{x_{1}+\cdots+x_{k}}{k},
\end{equation*}%
we get 
\begin{eqnarray*}
f\left( \frac{x_{1}+\cdots+x_{k}}{k}\right)  &=&f\left( \frac{%
x_{1}+\cdots+x_{k}+\left( 2^{n}-k\right) \frac{x_{1}+\cdots+x_{k}}{k}}{2^{n}}%
\right)  \\
&=&f\left( \frac{x_{1}+\cdots+x_{k}+\sum\limits_{j=k+1}^{2^{n}}\frac{x_{1}+\cdots+x_{k}}{k%
}}{2^{n}}\right)  \\
&=&\frac{f\left( x_{1}\right) +\cdots+f\left( x_{k}\right)
+\sum\limits_{j=k+1}^{2^{n}}f\left( \frac{x_{1}+\cdots+x_{k}}{k}\right) }{2^{n}} \\
&=&\frac{f\left( x_{1}\right) +\cdots+f\left( x_{k}\right) +\left(
2^{n}-k\right) f\left( \frac{x_{1}+\cdots+x_{k}}{k}\right) }{2^{n}},
\end{eqnarray*}%
whence 
\begin{equation*}
kf\left( \frac{x_{1}+\cdots+x_{k}}{k}\right) =f\left( x_{1}\right) +\cdots+f\left(
x_{k}\right), 
\end{equation*}%
which shows that $f$ is a Jensen function of $k$ variables.

\end{proof}

The main result of this paper reads as follows.

\begin{theorem}
Fix an integer $k \geq 2$ and a function $f:\left( 1,+\infty \right) \rightarrow
\left( 0,+\infty \right) $ \verb|[| $f:\left( 0,1 \right) \rightarrow \left(
0, +\infty \right) $ \verb|]|.
The following statements are pairwise equivalent:

\begin{enumerate}
	\item[(i)] the logarithmic Cauchy quotient function $L_{f,k}:\left( 1,+\infty
\right) ^{k}\rightarrow \left( 0,+\infty \right) $ \verb|[| $L_{f,k}:\left( 0,1
\right) ^{k}\rightarrow \left( 0,+\infty \right) $ \verb|]| is a $k$-variable mean in $%
\left( 1,+\infty \right) $ \verb|[| in $%
\left( 0,1 \right) $ \verb|]|;

\item[(ii)] there is a positive \verb|[| negative \verb|]| $c$ such that equality \eqref{eq:meanGenLfk} holds 
for all $x\in \left( 1,+\infty
\right)$ \verb|[| for all $x\in \left( 0,1
\right)$ \verb|]|;

\item[(iii)] the equality  
\begin{equation*}
L_{f,k}=\mathcal{L}_{k}
\end{equation*}%
holds in $ \left( 1,+\infty
\right)^k$ \verb|[| in $ \left( 0,1
\right)^k$ \verb|]|.
\end{enumerate}

 \end{theorem}

\begin{proof}
To prove the implication (i)$\Longrightarrow $(ii), assume that $L_{f,k}$ is
a mean in $\left( 1,+\infty \right) $. Fix arbitrarily $p>1$ and put $%
f_{0}:=f\mid _{_{\left[ p,p^{k}\right) }}.$ It follows from Remark \ref{PropertiesOfAMean} (ii) and Remark \ref{LfkReflexivity} that ,
the function $f:\left( 1,+\infty \right) \rightarrow \left( 0,+\infty \right) $
satisfies \eqref{eq:Lfk_Reflexivity}. Thus, by part (iii) of Lemma 2, for every $n\in \mathbb{Z},$

\begin{equation*}
f\left( x\right) =k^{-n}x^{^{\frac{k^{n}-1}{k-1}}}f_{0}\left(
x^{k^{n}}\right) ,\text{ \ \ \ \ \ }x\in \left[ p^{k^{-n}},p^{k^{-n+1}}%
\right).
\end{equation*}

Hence, for all $x_{1},\ldots ,x_{k}\in \left[ p^{k^{-n}},p^{k^{-n+1}}\right)
,$ we have 
\begin{equation*}
x_{1}\cdot \ldots \cdot x_{k}\in \left[ p^{k^{-(n-1)}},p^{k^{-(n-2)}}\right),
\end{equation*}%
and, by Definition \ref{def:Lfk}, for all $x_{1},\ldots ,x_{k}\in \left[
p^{k^{-n}},p^{k^{-n+1}}\right) ,$ 
\begin{eqnarray*}
L_{f,k}\left( x_{1},\ldots ,x_{k}\right) =
\frac{k^{-n}x_1^{\frac{k^{n-1}}{k-1}} f_0{(x_1^{k^n})}+\cdots +k^{-n}x_k^{\frac{k^{n-1}}{k-1}}f_0{(x_k^{k^n})}}{k^{-(n-1)}(x_1 \cdot \cdots \cdot x_k)^{\frac{k^{n-2}}{k-1}} f_0{((x_1 \cdot \cdots \cdot x_k)^{k^{n-1}})}}\\
=\frac{1}{k}\frac{%
\sum\limits_{j=1}^{k}x_{j}^{\frac{k^{n}-1}{k-1}}f_{0}\left( x_j^{k^{n}}\right) 
}{\left( \prod\limits_{j=1}^{k}x_{j}\right) ^{\frac{k^{n-2}}{k-1}%
}f_{0}\left( \left( \prod\limits_{j=1}^{k}x_{j}\right) ^{k^{n-1}}\right) }.
\end{eqnarray*}%
Since $L_{f,k}$ is a $k$-variable mean in the interval $I,\ $we have, for
all $x_{1},\ldots ,x_{k}\in \left[ p^{k^{-n}},p^{k^{-n+1}}\right) $,%
\begin{equation*}
\min \left( x_{1},\ldots ,x_{k}\right) \leq \frac{1}{k}\frac{%
\sum\limits_{j=1}^{k}x_{j}^{\frac{k^{n}-1}{k-1}}f_{0}\left( x_j^{k^{n}}\right) 
}{\left( \prod\limits_{j=1}^{k}x_{j}\right) ^{\frac{k^{n-2}}{k-1}%
}f_{0}\left( \left( \prod\limits_{j=1}^{k}x_{j}\right) ^{k^{n-1}}\right) }%
\leq \max \left( x_{1},\ldots ,x_{k}\right) .
\end{equation*}%
Choosing $y_{1},\ldots ,y_{k}\in \left[ p,p^{k}\right) ${\LARGE \ }%
arbitrarily, we have, for every $n\in \mathbb{Z}$, 
\begin{equation*}
x_{j}=y_{j}^{k^{-n}}\in \left[ p^{k^{-n}},p^{k^{-n+1}}\right) ~\text{\ \ \
for \ \ }j=1,\ldots ,k.
\end{equation*}%
Setting these numbers into the above inequalities, and, assuming that 
\begin{equation*}
y_{1}=\min \left( y_{1},...,y_{k}\right) \text{ \ \ \ and \ \ \ \ \ }%
y_{k}=\max \left( y_{1},...,y_{k}\right) ,
\end{equation*}%
(which can be done without any loss of generality), we get 
\begin{equation*}
y_{1}^{k^{-n}}\leq \frac{1}{k}\frac{\sum\limits_{j=1}^{k}y_{j}^{\frac{%
1-k^{-n}}{k-1}}f_{0}\left( y_{j}\right) }{\left(
\prod\limits_{j=1}^{k}y_{j}\right) ^{\frac{k^{-1}-k^{-n}}{k-1}}f_{0}\left(
\left( \prod\limits_{j=1}^{k}y_{j}\right) ^{k^{-1}}\right) }\leq
y_{k}^{k^{-n}},\text{ \ \ \ \ \ }y_{2},\ldots ,y_{k-1}\in \left[
p,p^{k}\right) .
\end{equation*}%
Letting here $n\rightarrow +\infty ,$ we obtain%
\begin{equation*}
1\leq \frac{1}{k}\frac{\sum\limits_{j=1}^{k}y_{j}^{\frac{1}{k-1}}f_{0}\left(
y_{j}\right) }{\left( \prod\limits_{j=1}^{k}y_{j}\right) ^{\frac{k^{-1}}{k-1}%
}f_{0}\left( \left( \prod\limits_{j=1}^{k}y_{j}\right) ^{k^{-1}}\right) }%
\leq 1,\text{ \ \ \ \ \ }y_{1},\ldots ,y_{k}\in \left[ p,p^{k}\right) ,
\end{equation*}%
whence, 
\begin{equation*}
\frac{1}{k}\sum\limits_{j=1}^{k}y_{j}^{\frac{1}{k-1}}f_{0}\left(
y_{j}\right) =\left( \left( \prod\limits_{j=1}^{k}y_{j}\right)
^{k^{-1}}\right) ^{\frac{1}{k-1}}f_{0}\left( \left(
\prod\limits_{j=1}^{k}y_{j}\right) ^{k^{-1}}\right) ,\ \ \ \ \ \
y_{1},\ldots ,y_{k}\in \left[ p,p^{k}\right) .
\end{equation*}%
Defining $g:\left[ a,a^{k}\right) \rightarrow \left( 0,+\infty \right) $ by%
\begin{equation*}
g\left( y\right) :=y^{\frac{1}{k-1}}f_{0}\left( y\right) ,\text{ \ \ \ \ \ }%
y\in \left[ p,p^{k}\right) ,
\end{equation*}%
we can write this equality as follows 
\begin{equation*}
g\left( \left( \prod\limits_{j=1}^{k}y_{j}\right) ^{k^{-1}}\right) =\frac{1}{%
k}\sum\limits_{j=1}^{k}g\left( y_{j}\right) ,\ \ \ \ \ \ y_{1},\ldots
,y_{k}\in \left[ p,p^{k}\right) .
\end{equation*}%
Since, for arbitrary $s_{j}\in \left[ \log p,\log p^{k}\right) ,$ $%
j=1,\ldots ,k,$ we have 
\begin{equation*}
y_{j}=e^{s_{j}}\in \left[ p,p^{k}\right) ,\text{ \ \ \ \ \ \ }j=1,\ldots ,k,
\end{equation*}%
we hence get 
\begin{equation*}
g\left( e^{\frac{1}{k}\left( s_{1}+\cdots +s_{k}\right) }\right) =\frac{1}{k}%
\left[ g\left( e^{s_{1}}\right) +\cdots +g\left( e^{s_{k}}\right) \right] ,\
\ \ \ \ \ s_{1},\ldots ,s_{k}\in \left[ \log p,\log{ p^{k}}\right) .
\end{equation*}%
Thus, the function 
\begin{equation*}
h:=g\circ \exp
\end{equation*}%
satisfies the Jensen functional equation 
\begin{equation*}
h\left( \frac{1}{k}\left( s_{1}+\cdots +s_{k}\right) \right) =\frac{1}{k}%
\left[ h\left( s_{1}\right) +\cdots +h\left( s_{k}\right) \right] ,\ \ \ \ \
\ s_{1},\ldots ,s_{k}\in \left[ \log p,\log {p^{k}}\right) .
\end{equation*}%
By \cite{Kuczma2}, p. 315, Theorem 1, and Lemma \ref{kJensen} there exists an additive function 
${a}:\mathbb{R}\rightarrow \mathbb{R}$ and $b\in \mathbb{R}$ such that%
\begin{equation*}
h\left( s\right) =\text{a}\left( s\right) +b,\text{ \ \ \ \ }\ s\in
\left[ \log a,\log a^{k}\right) .
\end{equation*}%
From the definitions of the functions $h,$ $g$ and $f_0$, we obtain
\begin{equation*}
g\left( y\right) =\text{a}\left( \log y\right) +b,\text{ \ \ \ \ \ }%
y\in \left[ p,p^{k}\right),
\end{equation*}%
and, using the $\mathbb{Q}$-homogeneity of the additive function $a$,
\begin{equation*}
f_{0}\left( y\right) =\frac{ a\left( \log{y} \right) +b }{y^{\frac{1}{k-1}}},\text{ \ \ \ }y\in \left[
p,p^{k}\right) .
\end{equation*}%
Hence, by Lemma 2 (iii), we have, for every $n\in \mathbb{Z},$ 
\begin{equation*}
f\left( x\right) =\frac{1}{x^{\frac{1}{k-1}}}\left( \text{a}\left(
\log x\right) +\frac{b}{k^{n}}\right) ,\text{ \ \ \ \ \ }x\in \left[
p^{k^{n}},p^{k^{n+1}}\right) .
\end{equation*}%
Setting this into equation \eqref{eq:Lfk_Reflexivity}, we get 
\begin{equation*}
\frac{1}{x^\frac{1}{k-1}} \left(a(\log{x})+\frac{b}{k^n}\right)=
\frac{1}{x^\frac{1}{k-1}} \left(a(\log{x})+\frac{b}{k^{n+1}}\right),
\end{equation*}
and thus
\begin{equation*}
b=0.
\end{equation*}%
Since $f$ is assumed to be positive, the function $a$ must be
continuous, i.e. there is $c>0$ such that%
\begin{equation*}
{a}\left( x\right) =cx,\text{ \ \ \ \ \ }x\in \mathbb{R}.
\end{equation*}%
Consequently, for every $n \in \mathbb{Z}$,%
\begin{equation*}
f\left( x\right) =\frac{c}{x^{\frac{1}{k-1}}}\log x,\text{ \ \ \ \ \ }x\in %
\left[ p^{k^{n}},p^{k^{n+1}}\right) .
\end{equation*}%
This proves the implication $(i)\Longrightarrow (ii).$

Assume $(ii)$ holds. Then, by Definition 1, we get, for all $x_{1},\ldots
,x_{k}\in \left( 1,+\infty \right) ,$%
\begin{eqnarray*}
L_{f,k}\left( x_{1},\ldots ,x_{k}\right) &=&\frac{\frac{c}{x_1^{\frac{1}{k-1}} }\log{x_1}+ \cdots + \frac{c}{x_k^{\frac{1}{k-1}}}\log{x_k}}{\frac{c}{(x_1 \cdots x_k)^\frac{1}{k-1} }\log{(x_1 \cdots x_k)}}\\
&=&
{(x_1 \cdots x_k)^\frac{1}{k-1} }
\frac{  \frac{\log{x_1}}{x_1^{\frac{1}{k-1}} }+ \cdots + \frac{\log{x_k}}{x_k^{\frac{1}{k-1}}} }{{\log{x_1}+ \cdots +\log{x_k}}}\\
&=&\frac{\sum\limits_{i=1}^{k}%
\left( \prod\limits_{j=1,j\neq i}^{k}x_{j}^{\frac{1}{k-1}}\right) \log x_{i}%
}{\sum\limits_{l=1}^{k}\log x_{l}}=\sum\limits_{i=1}^{k}\frac{\log x_{i}}{%
\sum\limits_{l=1}^{k}\log x_{l}}\left( \prod\limits_{j=1,j\neq
i}^{k}x_{j}\right) ^{\frac{1}{k-1}} \\
&=&\sum\limits_{i=1}^{k}\frac{\log x_{i}}{\sum\limits_{l=1}^{k}\log x_{l}}%
\mathcal{G}_{k-1}\left( x_{1},\ldots,x_{i-1},x_{i+1},\ldots,x_{k}\right) \\
&=&\mathcal{L}_{k}\left( x_{1},\ldots ,x_{k}\right) ,
\end{eqnarray*}%
where $\mathcal{L}_{k}:\left( 1,+\infty \right) ^{k}\rightarrow \left(
0,+\infty \right) $ is defined by formula \eqref{eq:Lfk_mean}, and $\mathcal{G}_{k-1}$ the 
$\left( k-1\right) $-variable geometric mean,%
\begin{equation*}
\mathcal{G}_{k-1}\left( x_{1},...,x_{i-1},x_{i+1},...,x_{k}\right) =\left(
\prod\limits_{j=1,j\neq i}^{k}x_{j}\right) ^{\frac{1}{k-1}}\text{, \ \ \ \ \ 
}i=1,...,{k}.
\end{equation*}%
For arbitrary $x_{1},\ldots ,x_{k}\in \left( 1,+\infty \right) $ put $x_{\min
}:=\min \left( x_{1},\ldots ,x_{k}\right) $ and $x_{\max }:=\left(
x_{1},\ldots ,x_{k}\right) $. Since 
\begin{eqnarray*}
x_{\min } &=&\sum\limits_{i=1}^{k}\frac{\log x_{i}}{\sum\limits_{l=1}^{k}%
\log x_{l}}x_{\min }\leq \sum\limits_{i=1}^{k}\frac{\log x_{i}}{%
\sum\limits_{i=l}^{k}\log x_{l}}\mathcal{G}_{k-1}\left(
x_{1},...,x_{i-1},x_{i+1},...,x_{k}\right) \\
&\leq &\sum\limits_{i=1}^{k}\frac{\log x_{i}}{\sum\limits_{l=1}^{k}\log x_{l}%
}x_{\max }=x_{\max },
\end{eqnarray*}%
we have $x_{\min }\leq \mathcal{L}_{k}\left( x_{1},\ldots ,x_{k}\right) \leq
x_{\max }$ (and these inequalities are strict if the $k$-tuple $\left(
x_{1},\ldots ,x_{k}\right) $ is not constant)\ which shows that $\mathcal{L}%
_{k}$ is a $k$-variable mean in $\left( 1,+\infty \right) $. Thus $%
(ii)\Longrightarrow (iii).$

The implication $(iii)\Longrightarrow (i)$ is obvious. This completes the
proof.
\end{proof}

\bigskip

\bigskip

In the context of Theorem 2 the natural question arises if it is
possible to extend the mean $\mathcal{L}_{k}$ onto $\left( 0,+\infty \right)
^{k}$. An answer gives the following

\begin{remark}
The function $\mathfrak{L}_{k}:\left( 0,+\infty \right) ^{k}\rightarrow
\left( 0,+\infty \right) $ defined by%
\begin{equation*}
\mathfrak{L}_{k}\left( x_{1},\ldots ,x_{k}\right) :=\left\{ 
\begin{array}{ccc}
\frac{\sum_{i=1}\mathcal{G}_{k-1}\left(
x_{1},...,x_{i-1},x_{i+1},...,x_{k}\right) \log x_{i}}{\sum_{l=1}^{k}\log x_{l}%
} & \text{if} & \left( x_{1},\ldots ,x_{k}\right) \in \left( 0,1\right)
^{k}\cup \left( 1,+\infty \right) ^{k} \\ 
1 & \text{if} & \left( x_{1},\ldots ,x_{k}\right) \notin \left( 0,1\right)
^{k}\cup \left( 1,+\infty \right) ^{k}%
\end{array}%
\right.
\end{equation*}%
is a $k$-variable mean in $\left( 0,+\infty \right) $, and it is the only
increasing extension of the means $\mathcal{L}_{k}:\left( 1,+\infty \right)
^{k}\rightarrow \left( 1,+\infty \right) $ and $\mathcal{L}_{k}:\left(
0,1\right) ^{k}\rightarrow \left( 0,1\right) .$
\end{remark}

\begin{proof}
By Theorem 2, the restriction $\mathfrak{L}_{k}|_{\left(
0,1\right) ^{k}}$ is a mean in $\left( 0,1\right) $, and $\mathfrak{L}%
_{k}|_{\left( 1,+\infty \right) ^{k}}$ is a mean in $\left( 1,\infty \right) 
$. \ If $\left( x_{1},...,x_{k}\right) \notin \left( \left( 0,1\right)
^{k}\cup \left( 1,+\infty \right) ^{k}\right) $ then,%
\begin{equation*}
\min \left( x_{1},...,x_{k}\right) \leq 1\leq \max \left(
x_{1},...,x_{k}\right) ,
\end{equation*}%
and, clearly, the number $1$ is the only possible value for an increasing
mean at such a point $\left( x_{1},...,x_{k}\right) .$

\end{proof}

To get an involutory counterpart of $\mathfrak{L}_{k},$ which could be
denoted by $\mathfrak{L}_{k}^{\limfunc{inv}}$, consider the following

\begin{remark}
Let $k\in \mathbb{N}$, $k\geq 2$. A function $M:\left( 1,+\infty \right)
^{k}\rightarrow \left( 1,+\infty \right) $ \verb|[| resp., $M:\left(
0,1\right) ^{k}\rightarrow \left( 0,1\right) $ \verb|]| is a $k$-variable mean in $(1,+\infty)$ \verb|[| resp. in $%
\left( 0,1\right) $ \verb|]| iff the function $M^{\limfunc{inv}}:\left( 0,1\right)
^{k}\rightarrow \left( 0,1\right) $ \verb|[| resp. $M^{\limfunc{inv}}:\left(
1,+\infty \right) ^{k}\rightarrow \left( 1,+\infty \right) $ \verb|]| defined by 
\begin{equation*}
M^{\limfunc{inv}}\left( x_{1},\ldots ,x_{k}\right) :=\frac{1}{M\left( \frac{1%
}{x_{1}},\ldots ,\frac{1}{x_{k}}\right) }
\end{equation*}%
is a $k$-variable mean in $\left( 0,1\right) $ \verb|[| resp. in $\left(
1,+\infty \right) $ \verb|]|$.$
\end{remark}

It easy to verify

\begin{remark}
The mean $\mathcal{L}_{k}^{\limfunc{inv}}:\left( 0,1\right) ^{k}\rightarrow
\left( 0,1\right) ,$ the involutory conjugate mean to $\mathcal{L}_{k},$ is
of the form%
\begin{equation*}
\mathcal{L}_{k}^{\limfunc{inv}}\left( x_{1},\ldots ,x_{k}\right) =\frac{%
\sum\limits_{i=1}^{k}\left( x_{i}\log x_{i}\right) \mathcal{G}_{k-1}\left(
x_{1},\ldots,x_{i-1},x_{i+1},\ldots,x_{k}\right) }{\sum\limits_{l=1}^{k}x_{l}\log x_{l}},%
\text{ \ \ \ \ \ }x_{1},\ldots ,x_{k}\in \left( 0,1\right) .
\end{equation*}
\end{remark}

Let us note some properties of the mean $\mathcal{L}_{k}$ in

\begin{proposition}
\begin{enumerate}
	\item[(i)] $\mathcal{L}_{k}$ is a symmetric and strict mean, but is neither
homogeneous nor translative.

\item[(ii)] $\mathcal{L}_{2}$ is the Beckenbach-Gini mean of generator $\log $,
i.e. 
\begin{equation*}
\mathcal{L}_{2}\left( x,y\right) =\frac{y\log x+x\log y}{\log x+\log y}\text{%
, \ \ \ \ \ }x,y \in (1,+\infty);
\end{equation*}%
and its involutory conjugate mean 
\begin{equation*}
\mathcal{L}_{2}^{\limfunc{inv}}\left( x,y\right) =xy\frac{\log x+\log y}{%
x\log x+y\log y}\text{, \ \ \ \ \ }x,y \in (0,1);
\end{equation*}

\item[(iii)] the bivariable geometric mean $\mathcal{G}$ is invariant with respect
to the mean-type mapping $\left( \mathcal{L}_{2}^{\limfunc{inv}},\mathcal{L}%
_{2}\right) $, i.e. $\mathcal{G\circ }\left( \mathcal{L}_{2}^{\limfunc{inv}},%
\mathcal{L}_{2}\right) =\mathcal{G}$, and the sequence $\left( \left( 
\mathcal{L}_{2}^{\limfunc{inv}},\mathcal{L}_{2}\right) ^{n}:n\in \mathbb{N}%
\right) $ of iterates of $\left( \mathcal{L}_{2}^{\limfunc{inv}},\mathcal{L}%
_{2}\right) $ converges uniformly on compact subsets of $\left( 1,+\infty
\right) ^{2}$ to $\left( \mathcal{G},\mathcal{G}\right) $ (see Theorem 1 in \cite{JM1999}%
). \ 
\end{enumerate}

\end{proposition}

\begin{example}
Indeed, for $k=2$, we have
\begin{equation*}
\mathcal{L}_{2}\left( 2,3\right) =\frac{3\log 2+2\log 3}{\log 2+\log 3}
=\frac{\log{72}}{\log{6}},
\end{equation*}
\begin{equation*}
\mathcal{L}_{2}\left( 2t,3t\right) =\frac{3t\log{ 2t}+2t\log{ 3t}}{\log {2t}+\log {3t}},
\end{equation*}
and
\begin{equation*}
\mathcal{L}_{2}\left( 2+t,3+t\right) =\frac{(3+t)\log{ (2+t)}+(2+t)\log{ (3+t)}}{\log {(2+t)}+\log {(3+t)}}.
\end{equation*}
Setting $t=2$, we  get $2\mathcal{L}_{2}\left( 2,3\right) =\frac{\log{144}}{\log 6} \neq \mathcal{L}_{2}\left( 4,6\right)=\frac{\log{5308416}}{\log{24}}$, and
$2+\mathcal{L}_{2}\left( 2,3\right)=\frac{\log{5184}}{6} \neq \mathcal{L}_{2}\left( 4,5\right)=\frac{\log{640000}}{\log{20}}$.
Thus $\mathcal{L}_{2}$ is neither homogeneous nor translative.
A similar argument gives (i) of Proposition 1.  
\end{example}

\section{A characterization of $\mathcal{L}_{k}$ with the aid of reflexivity
of $L_{f,k}$ and a special type of convexity of its generator}

Applying a generalized version of the Krull theorem on linear difference
equations (\cite{Krull}) given in Kuczma \cite{Kuczma} p. 114, Theorem 5.11), 
we give the following characterization of the
logarithmic Cauchy quotient mean $\mathcal{L}_{k}$.

\begin{theorem}
Let $k\in \mathbb{N}$, $k\geq 2,$ be fixed, and assume that $f:\left(
1,+\infty \right) \rightarrow \left( 0,+\infty \right) $ is differentiable and
such that the function $\log \circ f\circ \exp \circ \exp $ is convex. Then
the following conditions are pairwise equivalent:
\begin{enumerate}
	\item[(i)]
	\ the function $L_{f,k}$ is reflexive in $\left( 1,+\infty \right) ;$

\item[(ii)] \ there is $c>0$ such that 
$f$ is given by \eqref{eq:meanGenLfk} for all $x\in(1,+\infty)$;

\item[(iii)] $\ L_{f,k}=\mathcal{L}_{k}$.
\end{enumerate}

\end{theorem}

\begin{proof}
Assume (i). By Definition 1 and Remark 4, the function $f$ satisfies the
iterative functional equation:%
\begin{equation*}
f\left( x\right) =\frac{x}{k}f\left( x^{k}\right) ,\text{ \ \ \ \ \ }x \in (1,+\infty).
\end{equation*}%
Taking $\log $ on both sides gives us%
\begin{equation*}
\log f\left( x\right) =\log f\left( x^{k}\right) +\log x-\log k,\text{ \ \ \
\ \ }x\in (1,+\infty).
\end{equation*}%
Putting $t=\log{x}$ here we come to
the equivalent equality
\begin{equation*}
\log f\left( e^{t}\right) =\log f\left( e^{kt}\right) +t-\log k,\text{ \ \ \
\ \ }t\in (0,+\infty).
\end{equation*}%
Setting $g:\left( 0,+\infty \right) \rightarrow \mathbb{R}$, defined by%
\begin{equation*}
g=\log \circ f\circ \exp,
\end{equation*}%
we can write this equation in the form%
\begin{equation*}
g\left( t\right) =g\left( kt\right) +t-\log k\text{, \ \ \ \ \ }t\in (0,+\infty),
\end{equation*}%
that is%
\begin{equation*}
g\left( e^{\log t}\right) =g\left( e^{\log t+\log k}\right) +e^{\log t}-\log
k,\ \ \ \ \ t\in (0,+\infty).
\end{equation*}%
Setting $\tau =\log t$ we get 
\begin{equation*}
g\left( e^{\tau }\right) =g\left( e^{\tau +\log k}\right) +e^{\tau }-\log
k,\ \ \ \ \ \tau \in \mathbb{R}\text{,}
\end{equation*}%
and, consequently, the function $h:\mathbb{R\rightarrow R}$, defined by%
\begin{equation*}
h:=g\circ \exp =\log \circ f\circ \exp \circ \exp,
\end{equation*}%
satisfies the functional equation%
\begin{equation*}
h\left( \tau +\log k\right) =h\left( \tau \right) +\log k-e^{\tau },\ \ \ \
\ \tau \in \mathbb{R}\text{.}
\end{equation*}%
Differentiating both sides with respect to $\tau ,$ we obtain 
\begin{equation*}
h^{\prime }\left( \tau +\log k\right) =h^{\prime }\left( \tau \right)
-e^{\tau },\ \ \ \ \ \tau \in \mathbb{R}\text{.}
\end{equation*}%
Put%
\begin{equation*}
\text{\ }F\left( \tau \right) :=-e^{\tau },\text{ \ \ \ }\tau \in \mathbb{R}%
\text{.}
\end{equation*}%
Note that $F$ is concave, and%
\begin{equation*}
\lim_{\tau \rightarrow -\infty }\left[ F\left( \tau +\log k\right) -F\left(
\tau \right) \right] =\lim_{\tau \rightarrow -\infty }\left[ -e^{\tau +\log
k}-\left( -e^{\tau }\right) \right] =\lim_{\tau \rightarrow -\infty }\left[
e^{\tau }\left( -k+1\right) \right] =0.
\end{equation*}%
Therefore, in view of the theorem of Krull (\cite{Kuczma}, p. 114, Theorem
5.11), there exists exactly one, up to an additive constant, convex solution 
$h^{\prime }:\mathbb{R\rightarrow R}$ of the functional equation%
\begin{equation*}
h^{\prime }\left( \tau +\log k\right) =h^{\prime }\left( \tau \right)
+F\left( \tau \right) ,\ \ \ \ \ \tau \in \mathbb{R}\text{.}
\end{equation*}%
It is easy to verify that, if $f$ is given by formula \eqref{eq:meanGenLfk} in part (ii), then 
$h=\log \circ f\circ \exp \circ \exp $ satisfies this equation, as
\begin{equation*}
\log \circ f\circ \exp \circ \exp (\tau)=\log{c}+\tau -\frac{1}{k-1} e^{\tau}, \quad \tau \in \mathbb{R}.
\end{equation*}
Since $%
\left( \log \circ f\circ \exp \circ \exp \right) ^{\prime }$ is decreasing,
the function $\log \circ f\circ \exp \circ \exp $ is concave. 
Indeed, we have
\begin{equation}
(\log \circ f\circ \exp \circ \exp (\tau))^{\prime \prime} =-\frac{1}{k-1}e^{\tau}, \quad \tau \in \mathbb{R},
\end{equation}
implying the concavity of the function $\log \circ f\circ \exp \circ \exp $.
Thus we have
shown (ii). Since logarithmic Cauchy quotients for a given generator $f$ are
uniquely determined, the implication (ii)$\Rightarrow $(iii) follows. The
remaining implication is due to part (ii) of Remark 1. This finishes the
proof.
\end{proof}

Weakening the assumption on $\mathcal{L}_{f,k}$ while adding some regularity assumption on the generator $f$,
and making use of the idea applied in \cite{JM1972}, one gets the
following characterization of the logarithmic Cauchy mean. 

\begin{theorem}
Let $k\in \mathbb{N}$, $k\geq 2$ be fixed. Assume that $f:\left( 1,+\infty
\right) \rightarrow \left( 0,+\infty \right) $ is such that, for some $c>0,$
the function 
\begin{equation}
\left( 0,+\infty \right) \ni x\longmapsto \frac{f\left( x\right) -c\left(
x-1\right) }{\left( x-1\right) ^{2}}  \label{eq:BdLfk}
\end{equation}%
is bounded in a right vicinity of $1$.

Then the following conditions are pairwise equivalent
\begin{enumerate}
	\item[(i)] \ the function $L_{f,k}$ is reflexive in $\left( 1,+\infty \right) ;$

\item[(ii)] \ there is $c>0$ such that $f$satisfies \eqref{eq:meanGenLfk} for all $x\in (1,+\infty)$;

\item[(iii)] $\ L_{f,k}=\mathcal{L}_{k}$.
\end{enumerate}
\end{theorem}

\begin{proof}
From \eqref{eq:BdLfk} we have 
\begin{equation}
f\left( x\right) =c\left( x-1\right) +\varphi \left( x\right) \left(
x-1\right) ^{2}\text{, \ \ \ \ \ }x\in (1,+\infty),  \label{eq:f(x)}
\end{equation}%
where the function $\varphi :\left( 1,+\infty \right) \rightarrow \mathbb{R}$
defined by 
\begin{equation*}
\varphi \left( x\right) :=\frac{f\left( x\right) -c\left( x-1\right) }{%
\left( x-1\right) ^{2}}\text{, \ \ \ \ \ }x\in (1,+\infty)\text{,}
\end{equation*}%
is bounded in an interval $\left( 1,1+r\right) $, for some $r>0.$

Assume (i). In view of Remark 4, the generator $f$ of $L_{f,k}$ satisfies
the functional equation \eqref{eq:Lfk_Reflexivity},
that is equivalent to the functional equation%
\begin{equation}
f\left( x\right) =\frac{k}{x^{\frac{1}{k}}}f\left( x^{\frac{1}{k}}\right) 
\text{, \ \ \ \ \ \ }x \in (1, +\infty).
\label{eq:LfkReflexivity2}
\end{equation}%
Taking into account \eqref{eq:f(x)}, we conclude that $\varphi $  satisfies the
functional equation 
\begin{equation*}
c\left( x-1\right) +\varphi \left( x\right) \left( x-1\right) ^{2}=\frac{k}{%
x^{\frac{1}{k}}}\left[ c\left( x^{\frac{1}{k}}-1\right) +\left( x^{\frac{1}{k%
}}-1\right) ^{2}\varphi \left( x^{\frac{1}{k}}\right) \right] \text{, \ \ \
\ \ \ }x \in (1, +\infty),
\end{equation*}%

which can be written in the form

\bigskip

\begin{equation}
\varphi \left( x\right) =\frac{c\left( 1+k-x-kx^{-\frac{1}{k}}\right) }{%
\left( x-1\right) ^{2}}+kx^{-\frac{1}{k}}\left( \frac{x^{\frac{1}{k}}-1}{x-1}%
\right) ^{2}\varphi \left( x^{\frac{1}{k}}\right) ,  \label{eq:BoundGenLfk}
\end{equation}%
and, moreover, $\varphi $ is bounded in an interval $\left( 1,1+r\right) $.

Assume that the functions $\varphi _{1},\varphi _{2}:\left( 1,+\infty
\right) \rightarrow \mathbb{R}$ are bounded in $\left( 1,1+r\right) $ for
some $r>0$, and satisfy equation \eqref{eq:BoundGenLfk}, that is%
\begin{equation*}
\varphi _{i}\left( x\right) =\frac{c\left( 1+k-x-kx^{-\frac{1}{k}}\right) }{%
\left( x-1\right) ^{2}}+kx^{-\frac{1}{k}}\left( \frac{x^{\frac{1}{k}}-1}{x-1}%
\right) ^{2}\varphi _{i}\left( x^{\frac{1}{k}}\right) \text{, \ \ \ \ \ }x\in (1,+\infty);%
\text{ \ \ }i=1,2.
\end{equation*}%
Hence, putting 
\begin{equation*}
\psi :=\left\vert \varphi _{1}-\varphi _{2}\right\vert \text{ \ \ \ \ and \
\ \ \ }\alpha \left( x\right) :=x^{\frac{1}{k}}\text{ \ for \ }x\in (1,+\infty),\text{ \ }
\end{equation*}%
we see that $\psi $ is nonnegative and bounded solution of the functional
equation%
\begin{equation}
\psi \left( x\right) =kx^{-\frac{1}{k}}\left( \frac{x^{\frac{1}{k}}-1}{x-1}%
\right) ^{2}\psi \left( \text{\ }\alpha \left( x\right) \right) \text{, \ \
\ \ }x\in (1,+\infty).  \label{eq:BoundGenLfk!}
\end{equation}%
Note that

\begin{eqnarray*}
\frac{x^{\frac{1}{k}}-1}{x-1} &=&\frac{x^{\frac{1}{k}}-1}{\left( x^{\frac{1}{%
k}}\right) ^{k}-1}=\frac{x^{\frac{1}{k}}-1}{\left( x^{\frac{1}{k}}-1\right)
\left( \left( x^{\frac{1}{k}}\right) ^{k-1}+\left( x^{\frac{1}{k}}\right)
^{k-2}+\cdots+x^{\frac{1}{k}}+1\right) } \\
&=&\frac{1}{\left( x^{\frac{1}{k}}\right) ^{k-1}+\left( x^{\frac{1}{k}%
}\right) ^{k-2}+\cdots+x^{\frac{1}{k}}+1},
\end{eqnarray*}%
so, for all $x\in (1,+\infty)$, we have 
\begin{equation*}
kx^{-\frac{1}{k}}\left( \frac{x^{\frac{1}{k}}-1}{x-1}\right) ^{2}=\frac{kx^{-%
\frac{1}{k}}}{\left( \left( x^{\frac{1}{k}}\right) ^{k-1}+\left( x^{\frac{1}{%
k}}\right) ^{k-2}+\cdots+x^{\frac{1}{k}}+1\right) ^{2}}
\end{equation*}%
Hence 
\begin{equation*}
\lim_{x\rightarrow 1}kx^{-\frac{1}{k}}\left( \frac{x^{\frac{1}{k}}-1}{x-1}%
\right) ^{2}=\lim_{x\rightarrow 1}\frac{kx^{-\frac{1}{k}}}{\left( \left( x^{%
\frac{1}{k}}\right) ^{k-1}+\left( x^{\frac{1}{k}}\right) ^{k-2}+\cdots+x^{\frac{%
1}{k}}+1\right) ^{2}}=\frac{k}{k^{2}}=\frac{1}{k},
\end{equation*}%
and, as $k\geq 2$, \ there is $r>0$ such that%
\begin{equation*}
kx^{-\frac{1}{k}}\left( \frac{x^{\frac{1}{k}}-1}{x-1}\right) ^{2}\leq \frac{1%
}{2}\text{, \ \ \ \ \ \ }x\in \left( 1,1+r\right) .
\end{equation*}%
Since $\alpha \left( \left( 1,1+r\right) \right) \subset \left( 1,1+r\right) 
$, in view of \eqref{eq:BoundGenLfk!},%
\begin{equation*}
0\leq \psi \left( x\right) \leq \frac{1}{2}\psi \left( \text{\ }\alpha
\left( x\right) \right) \text{, \ \ \ \ }x\in \left( 1,1+r\right) ,
\end{equation*}%
the boundedness of $\psi $ implies that%
\begin{equation*}
\psi \left( x\right) =0\text{, \ \ \ \ \ \ }x\in \left( 1,1+r\right) .
\end{equation*}%
Now, from \eqref{eq:BoundGenLfk!}, taking into account that 
\begin{equation*}
\lim_{n\rightarrow \infty }\alpha ^{n}\left( x\right) =1\text{, \ \ \ \ }x\in (1,+\infty)%
\text{,}
\end{equation*}%
we conclude that $\psi \left( x\right) =0$ for every $x \in(0,+\infty)$ which shows that $%
\varphi _{1}=\varphi _{2}$. This proves that there is at most one solution
of equation \eqref{eq:LfkReflexivity2} satisfying condition \eqref{eq:BdLfk}.

Now the implication $(i)\Longrightarrow (ii)$ follows from the fact that the
function 
\begin{equation*}
\left( 1,+\infty \right) \ni x\longmapsto \frac{c\log x}{\sqrt[k-1]{x}},
\end{equation*}%
is a solution of the reflexivity equation \eqref{eq:LfkReflexivity2} and satisfies condition \eqref{eq:BdLfk}.

The remaining implications are obvious.
\end{proof}

Since twice continuously differentiable functions satisfy condition \eqref{eq:BdLfk}, the
following result is an immediate consequence of the above result.

\begin{corollary}
Let $k\in \mathbb{N}$, $k\geq 2$ be fixed. Assume that $f:\left( 1,+\infty
\right) \rightarrow \left( 0,+\infty \right) $ is of the class $C^{2}$ and
the function 
\begin{equation*}
\left( 1,+\infty \right) \ni x\longmapsto f\left( x\right)
\end{equation*}%
has an extension that is of the class $C^{2}$ in the interval $\left[
1,+\infty \right) $.

Then the following conditions are pairwise equivalent:
\begin{enumerate}
	\item[(i)] \ the function $L_{f,k}$ is a premean in $\left( 1,+\infty \right) ;$

\item[(ii)] \ there is $c>0$ such that 
$f$ satisfies \eqref{eq:meanGenLfk} for all $x\in (1,+\infty)$

\item[(iii)] $\ L_{f,k}=\mathcal{L}_{k}$.
\end{enumerate}

\end{corollary}

\begin{acknowledgement}
The authors are greatly indebted to the referees for careful reading the
manuscript and giving valuable suggestions.
\end{acknowledgement}

\bigskip

\end{document}